\documentclass[10pt]{amsart}

\usepackage{amssymb,amsmath,latexsym}
\usepackage[varg]{pxfonts}
\usepackage{amsmath, amsthm, amscd, amsfonts, amssymb, graphicx, color}
\usepackage{marvosym}
\usepackage{paralist}
\usepackage{color}
\title[   ]{ Existence of  group nonexpansive retractions  and ergodic  theorems  in    topological groups}
 \author{  Ebrahim  Soori$^{*,1}$,  Ravi   Agarwal$^{2}$  $\,\,$ and Donal O'Regan$^3$       }
 \thanks{ \!\! \!\! \!\! \!\!1 Corresponding author  
  \\ E-mail addresses:  Ravi.Agarwal@tamuk.edu(Ravi   Agarwal),  donal.oregan@nuigalway.ie (Donal O'Regan),   sori.e@lu.ac.ir, sori.ebrahim@yahoo.com (E. Soori).}

\theoremstyle{plain}
\newtheorem{lem}{\textbf{Lemma}}[section]
\newtheorem{thm}[lem]{\textbf{Theorem}}
\newtheorem{co}[lem]{\textbf{Corollary}}

\newtheorem{de}[lem]{\textbf{Definition}}
\newtheorem{ex}[lem]{\textbf{Example}}

\theoremstyle{definition}

\theoremstyle{definition}
\theoremstyle{remark}

\renewcommand{\sc}{\mathcal{S}}

\begin{document}
\begin{large}

\maketitle
\noindent
\\
$^1$Department  of Mathematics, Lorestan University,  P. O. Box 465, Khoramabad, Lorestan, Iran.\\
$^2$Department of Mathematics, Texas A and M University-Kingsville, Texas 78363, USA.\\
$^3$School of Mathematics, Statistics  and Applied Mathematics, National University of Ireland, Galway, Ireland.\\

\vspace{2mm}

\begin{abstract}

\begin{normalsize}
 Suppose that  $G$ is a    topological group and  $ C $       a      compact subset                 of $G$. In this paper we define group nonexpansive mappings and then  we consider $\sc = \{T_{i} : i \in I \}$ as a family of the  group  nonexpansive mappings on $C$.    Also  we study  the existence of  group nonexpansive retractions  $P_{i}$ from $C$ onto
$\text{Fix}(\sc)$   such that $P_{i}T_{i} = T_{i}P_{i} = P_{i}$.

\end{normalsize}
\end{abstract}
\begin{normalsize}
   \end{normalsize}

\section{ Introduction}
  A topological group $G$ is a set endowed with two structures, a group
structure and a topological structure. Specifically, $G$ is both an abstract group and
a topological space such that the two maps
\begin{align*}
  G\times G \rightarrow G:  (x,y) \mapsto xy\\
   G \rightarrow G:  x \mapsto x^{-1}\\
\end{align*}
are assumed to be continuous. Also, the Hausdorff condition will be imposed.

Let $D$ be a subset of $B$ where $B$ is a subset of a  topological group $G$. A mapping  $P$  is called  a retraction of $B$ onto $D$,  if
$Px = x$ for each $x \in D$.

The first nonlinear ergodic theorem for nonexpansive mappings in a Hilbert space was established by Baillon \cite{bia}: Let $C$
be a nonempty closed convex subset of a Hilbert space $H$ and let $T$ be a nonexpansive mapping of $C$ into itself. If the set
$\text{Fix}(T)$ of fixed points of $T$ is nonempty, then for each $x \in C$, the Cesaro means
$
S_{n}x=\frac{1}{n}\sum_{k=1}^{n}T^{k}x
$
converge weakly to some $y \in \text{Fix}(T)$. In Baillon's theorem, putting $y = Px$ for each $x \in C$, $P$ is a nonexpansive retraction of $C$
onto $\text{Fix}(T)$ such that $PT^{n} = T ^{n} P = P$ for all positive integers $n$ and $Px \in \overline{co}\{T ^{n}x: n = 1, 2, . . .\}$ for each $x \in C$. Takahashi \cite{tak1}
proved the existence of such retractions, “ergodic retractions”, for non-commutative semigroups of nonexpansive mappings
in a Hilbert space: If $S$ is an amenable semigroup, $C$ is a closed, convex subset of a Hilbert space $H$ and $\sc= \{T_{s}: s \in S\}$
is a nonexpansive semigroup on $C$ such that $\text{Fix}(\sc)\neq \emptyset$, then there exists a nonexpansive retraction $P$ from $C$ onto $\text{Fix}(\sc)$
such that $PT_{t} = T _{t} P = P$ for each $t \in S$  and $Px \in \overline{co}\,\{T_{t}x:t\in S\}$    for each $x \in C$. These results were extended to uniformly
convex Banach spaces for commutative semigroups in \cite{hi}  and for non-commutative amenable semigroups in \cite{lst, lnt} and recently  for a family of $Q$-nonexpansive mappings in  locally convex spaces in \cite{soori}. For
other results we refer the reader to  \cite{said2012, said11, said, said33}.

 In this paper,  first we define group
nonexpansive mappings. Then  we establish some  ergodic retractions  for  topological groups, based on the definition. Also we present a family of desired retractions by removing  "convexity" in the above mentioned theorems  in the  topological group setting.
\section{preliminaries }

In this section,  we introduce our definition and give some examples:
\begin{de}
 Suppose that   $G$     is a topological group  and  $ C \subset G $.    A mapping $T : C  \rightarrow C$ is said to be
    \textbf{group nonexpansive} if for   each  $x, y \in C$  and each closed  neighborhood $U \in \mathfrak{B}_{e}$ (where $ \mathfrak{B}_{e}$ is a local base in $e$ (identity element)) that $ xy^{-1} \in U$ then we have  $Tx (Ty)^{-1} \in U$.
\end{de}
\begin{ex}
Our group nonexpansiveness is more general than    nonexpansiveness i.e.  every   nonexpansive mapping   is a
  group nonexpansive mapping if the topological group is a normed vector space. Indeed, let $E$ be a normed vector space and  $ C \subset E $. Consider the  closed  neighborhood $U=\{z \in E: \|z\|\leq r \} \in \mathfrak{B}_{0}$ (where $ \mathfrak{B}_{0}$ is a local base in $0$)  that $ x-y \in U$ for a positive number $r$. If $T$ is a nonexpansive mapping on $C$ then we have
 \begin{equation*}
    \|T x -T y\|\leq \|x-y\|,
  \end{equation*}
 and  hence $ Tx-Ty \in U$. Thus $T$ is a group nonexpansive.
\end{ex}
In the following example we consider a case that  some  group nonexpansive mappings are    nonexpansive mapping.
  \begin{ex}
  Let $G$ be metric    topological group with a right invariant  metric (that is, $d(yx ,  zx) = d(y , z)$ for all $x, y, z  \in G$).  Let   $    T$ be a group nonexpansive mapping from $G$ into $G$.   Consider the neighbourhoods $N_{d(xy^{-1},e)+\frac{1}{n}}(e)$ of $e$ with  $xy^{-1} \in N_{d(xy^{-1},e)+\frac{1}{n}}(e)$ for each $n \in \mathbb{N}$.  Then we have $Tx (Ty)^{-1} \in N_{d(xy^{-1},e)+\frac{1}{n}}(e)$ for each $n \in \mathbb{N}$, and hence
   \begin{equation*}
    d(T x  (T y)^{-1}, e) \leq d(xy^{-1}, e) + \frac{1}{n},
  \end{equation*}
 for each $n \in \mathbb{N}$. Therefore
  \begin{equation*}
    d(T x  (T y)^{-1}, e)\leq d(xy^{-1}, e),
  \end{equation*}
  and since $d$ is right invariant, we have
  \begin{equation*}
    d(T x ,T y)\leq d(x,y).
  \end{equation*}
 Then
  we conclude that  $T$  is a nonexpansive mappings in the sense of nonexpansive mappings in metric spaces.
 \end{ex}

\section{\textbf{Ergodic retractions for families of group  nonexpansive mappings on topological groups}}
Let   $G$ be a topological group. In this section, we study the existence of  group nonexpansive retractions onto the set of common fixed points of a family of
 group nonexpansive mappings that commute with the mappings. A  group nonexpansive retraction that commutes with the mappings is usually called an ergodic retraction.

First, we prove the following theorem
which is the main result of this section and  will be essential in the sequel.
\begin{thm}\label{shlp}
Let $G$ be a topological group and let $ C $   be   a      compact subset                 of $G$.
     Suppose that $\sc = \{T_{i} : i \in I \}$ is a family of the  group  nonexpansive mappings on $C$ such that
$\text{Fix}(\sc)\neq \emptyset$ and for every $\alpha \in I$,   there exists a subnet $\{T_{\alpha}^{n_{\gamma}}\}$ of the sequence $\{T_{\alpha}^{n}\}$ such that $\displaystyle\lim _{\gamma}T_{\alpha}^{n_{\gamma}}x=\displaystyle\lim _{\gamma}T_{\alpha}^{n_{\gamma}-1}x$ for each $ x \in C$.   Also suppose
   for every nonempty    compact   $\sc $-invariant subset $K$ of $C$,
$K \cap \text{Fix}(\sc) \neq \emptyset$.
Then, for each $i \in I$, there exists a  group nonexpansive retraction $P_{i}$ from $C$ onto
$\text{Fix}(\sc)$, such that $P_{i}T_{i} = T_{i}P_{i} = P_{i}$  and every   closed   $\sc$-invariant
subset of $C$ is also $P_{i}$-invariant.
\end{thm}
\begin{proof}

 Let  $C^{C}$ be the product space with  the product topology induced by the   relative  topology  on  $C$. Now for a fixed $\alpha \in I$, consider the following set

  $\mathfrak{R}=\{T \in C^{C}: T \; \text{is group nonexpansive}, T\circ T_{\alpha}=T\\
  \text{   } \text{ and every    closed   $\sc$-invariant subset of $C$ is also $T$ -invariant}\}$.\\
   From the fact that $G$ is Hausdorff,   for each $z \in \text{Fix}(\sc)$, the singleton set  $\{z\}$ is a    closed   $\sc$-invariant subset of $C$,  and  then for each    $T \in \mathfrak{R}$, $Tz=z$. Fix $z_{0} \in \text{Fix}(\sc)$ and let for each $x \in C$,
\begin{align*}
C_{x}:=&\{y \in C :      \text{for each closed neighborhood } U \;  \text{of}\; e \; \text{that} \; x z_{0}^{-1} \in U \\& \text{then} \; yz_{0}^{-1}  \in U\}.
\end{align*}
      For all   $x \in C$ and $T \in \mathfrak{R}$, we have that  $T(x) \in  C_{x}$. Since $T$ is group nonexpansive for a closed neighborhood $U$ of $e$, if $ xz_{0}^{-1} \in U$  then   $T(x)z_{0}^{-1} = T(x)(T(z_{0}))^{-1}\in U$. Hence $\mathfrak{R} \subseteq \prod_{x \in C} C_{x}$, where  $\prod_{x \in C} C_{x}$ is the Cartesian product of sets $C_{x}$ for all  $x \in C$.  Let $\{y_{\beta}\}$ be a net in $C_{x}$  such that $y_{\beta}\rightarrow y$. Consider   a closed neighborhood $U$  of $e$ such that $x z_{0}^{-1} \in U$. Then we have $y_{\beta}z_{0}^{-1}\in U$.  From the fact that the mapping $(x, y) \mapsto xy $ is continuous we have $y_{\beta}z_{0}^{-1} \rightarrow yz_{0}^{-1}$    and since $U$ is closed we conclude that $yz_{0}^{-1}\in  U$, and therefore $C_{x}$ is closed  and    since $C$ is     compact  we conclude that  $C_{x}$ is    compact. By Tychonoff's theorem, we know that  when $C_{x}$ is given the  relative   topology and $ \prod_{x \in C} C_{x}$ is given the corresponding product topology, $ \prod_{x \in C} C_{x}$ is   compact. Next we prove that  $\mathfrak{R}$ is   closed in $\prod_{x \in C} C_{x}$. Let $\{ T_{\lambda}: \lambda \in \Lambda\}$ be a net in $\mathfrak{R}$ which  converges to $T_{0}$ in  $ \prod_{x \in C} C_{x}$. Hence if $z \in \text{Fix}(\sc)$, then we have  $T_{\lambda}z=z$ for each $ \lambda \in \Lambda$ (because  $T_{\lambda} \in \mathfrak{R})$ and   $T_{0}z=  \displaystyle \lim_{\lambda}T_{\lambda}(z)=z$. From the fact that the mapping $(x, y) \mapsto xy $ and $x \mapsto x^{-1}  $  are continuous, if we consider  a  closed  neighborhood $U$ of $e$ that  $x y^{-1} \in U$
        then we have $T_{0}x(T_{0}y)^{-1} = \displaystyle\lim_{\lambda}  T_{\lambda}x(T_{\lambda}y)^{-1}\in  U$. Hence,  $T$ is group nonexpansive. Obviously, we have  $T_{0}\circ T_{\alpha}=T_{0}$  and every   closed   $\sc$-invariant
subset of $C$ is also $T_{0}$-invariant. Therefore, $T_{0} \in \mathfrak{R}$. Then $\mathfrak{R}$ is closed in $\prod_{x \in C} C_{x}$.
Since $\prod_{x \in C} C_{x}$ is compact, hence   $\mathfrak{R}$ is compact.  Next, we show that  $\mathfrak{R}\neq \emptyset$.
 Consider the mappings $S_{n}=T^{n-1}_{\alpha} \in \prod_{x \in C} C_{x}$ for each $n \in \mathbb{N}$. Then from the fact that     $\prod_{x \in C} C_{x}$ is compact and using our condition,
  it has a   convergent subnet $\{S_{n_{\eta}}\}$ such that $\displaystyle\lim_{\eta} T_{\alpha}^{n_{\eta}} x =\displaystyle\lim_{\eta} T_{\alpha}^{n_{\eta}-1} x $ for each $x \in C$.
Define for each $x \in C$, $T (x) = \lim_{\eta} S_{n_{\eta}}x$. We now check that $T \in \mathfrak{R}$. Note that, from  the    continuity
of    the mapping $(x, y) \mapsto xy $ and $x \mapsto x^{-1}  $ and the group nonexpansiveness of $S_{n_{\eta}}$ and closedness of $U$, $T$ is  group nonexpansive. Indeed, if $xy^{-1}\in U$ for any closed neighbourhood $U$ of $e$ and $x,y \in C$, then we have  $\displaystyle Tx(Ty)^{-1}=\lim_{\eta} S_{n_{\eta}}x(\lim_{\eta} S_{n_{\eta}}y)^{-1}=\lim_{\eta}S_{n_{\eta}}x( S_{n_{\eta}}y)^{-1} \in U$.
 Moreover, $T (T_{\alpha}x) = \displaystyle\lim_{\eta} S_{n_{\eta}}(T_{\alpha} x) =\displaystyle\lim_{\eta} T_{\alpha}^{n_{\eta}} x =\displaystyle\lim_{\eta} T_{\alpha}^{n_{\eta}-1} x=
\displaystyle\lim_{\eta} S_{n_{\eta}}( x) = T (x)$.
Finally, if $D$ is a closed $\sc$-invariant
subset of $C$, it is clear that $D$ is $S_{n_{\eta}}$-invariant and thus from the closedness of $D$, is $T$-invariant. Therefore, we have shown that $T\in \mathfrak{R}\neq \emptyset$.

Now define a preorder $\preceq$ in $\mathfrak{R}$ by $T_{1} \preceq T_{2}$ if for each $U \in \mathfrak{B}_{e} $  that $ T_{2}x(T_{2}y)^{-1} \in U$   we have $T_{1}x (T_{1}y)^{-1}\in U$, and by  using a method similar to Bruck's method \cite{broook}, we find a minimal element $T_{min}$ in  $\mathfrak{R}$. Indeed, using Zorn's Lemma, it is enough that we   show that each  linearly ordered subset of $\mathfrak{R}$ has
a lower bound in  $\mathfrak{R}$. Let $\{ A_{\lambda}\}$ be a linearly ordered subset of $\mathfrak{R}$. Then the family of sets  $\{T \in \mathfrak{R}:  T \preceq A_{\lambda}\}$ is a linearly ordered  subset of  $\mathfrak{R}$  by inclusion. Taking  into account  the  closeness proof of  $\mathfrak{R}$  in
 $\prod_{x \in C} C_{x}$, these sets are closed in  $\mathfrak{R}$,  and
hence   compact.
  Then from the finite intersection property, there exists  $R \in\bigcap_{\lambda}\{T \in \mathfrak{R}:  T \preceq A_{\lambda}\}$ with $R \preceq A_{\lambda}$ for all $\lambda$. Then each  linearly ordered subset of $\mathfrak{R}$ has
a lower bound in  $\mathfrak{R}$. We have
shown  that there exist    a minimal element $P_{\alpha}$  in the following sense:

 \text{if }$ T \in \mathfrak{R} \; \text{and }  \text{for each }  U \in \mathfrak{B}_{e} \text{ that }  P_{\alpha}x(P_{\alpha}y)^{-1} \in U \;   \text{then }   Tx (Ty)^{-1}\in U, \\
\text{  }   \text{ then }   \text{for each }  U^{\prime} \in \mathfrak{B}_{e} \text{ that }    Tx (Ty)^{-1} \in  U^{\prime}   \text{ we have }  P_{\alpha}x(P_{\alpha}y)^{-1} \in U^{\prime}$. \quad ($*$)

  Next we prove that $P_{\alpha}x \in \text{Fix}(\sc)$  for every    $x  \in C $. For a given   $x  \in C $, consider $K:=\{T(P_{\alpha}x): T \in \mathfrak{R}\}$.
   From the fact that  $\mathfrak{R}$  is     compact,  from   Proposition 3.3.18 and  Definition 3.3.19 in  \cite{runde}, we conclude that    $K$ is a nonempty    compact   subset of $C$. Now we have $S(K)\subset K$ for each $S \in \sc$, because  $STT_{\alpha}=ST$  for each  $T \in \mathfrak{R}$ hence $ST \in \mathfrak{R}$ i.e,  $K$ is $\sc$-invariant.

    From our assumption $K \cap \text{Fix}(\sc) \neq \emptyset$. Then there exists $L \in \mathfrak{R}$
such that  $L(P_{\alpha}x) \in \text{Fix}(\sc)$. Suppose that  $y=L(P_{\alpha}x)$. Since $P_{\alpha}, L \in \mathfrak{R}$ and  the set  $\{y\}$ is $\sc$-invariant,  we have $P_{\alpha}(y)=L(y)=y$, and since $L$ is group nonexpansive,  $P_{\alpha}$ is minimal and      $   L(P_{\alpha}x)(L(P_{\alpha}y) )^{-1}  =  L(P_{\alpha}x)  y^{-1}=yy^{-1}=e \in U$, for each $U\in  \mathfrak{B}_{e}$ and then we have $ (P_{\alpha}x) y^{-1}=  P_{\alpha}x (P_{\alpha}y)^{-1} \in U$  for each $U\in  \mathfrak{B}_{e}$ (to see this consider $ LP_{\alpha}$ instead of $T$ in ($*$))    and by (vi) of Corollary 1.11 in \cite{sob}, $( P_{\alpha}x) y^{-1}=e$, hence  $P_{\alpha}x=y \in  \text{Fix}(\sc)$ and this   holds  for each $x \in C$.

 Since     $P_{\alpha} \in \mathfrak{R}$,  $T_{\alpha} \in \sc$   and  $\{P_{\alpha}x\}$ is $\sc$-invariant  for each $x \in C$, hence, it must be $P_{\alpha}$-invariant  for each $x \in C$. Then   we conclude that  $P_{\alpha}^{2}=P_{\alpha}$ and $P_{\alpha}T_{\alpha}=T_{\alpha}P_{\alpha}=P_{\alpha} $.

\end{proof}
As a consequence of Theorem \ref{shlp}, we establish an ergodic retraction by a  group nonexpansive retraction.
\begin{thm}\label{kgfdy}
Let $G$ be a topological group and let $ C $   be   a      compact subset                 of $G$.
     Suppose that $\sc = \{T_{i} : i \in I \}$ is a family of  group  nonexpansive mappings on $C$ such that
$\text{Fix}(\sc)\neq \emptyset$ and for every $\alpha \in I$,   there exists a subnet $\{T_{\alpha}^{n_{\gamma}}\}$ of the sequence $\{T_{\alpha}^{n}\}$ such that $\displaystyle\lim _{\gamma}T_{\alpha}^{n_{\gamma}}x=\displaystyle\lim _{\gamma}T_{\alpha}^{n_{\gamma}-1}x$ for each $ x \in C$.    Also suppose
   for every nonempty    compact   $\sc $-invariant subset $K$ of $C$,
$K \cap \text{Fix}(\sc) \neq \emptyset$.
 If there is a group nonexpansive retraction $R$ from $C$ onto
$\text{Fix}(\sc)$,
then for each $i \in I$, there exists a group nonexpansive retraction $P_{i}$ from $C$ onto
$\text{Fix}(\sc)$, such that $P_{i}T_{i} = T_{i}P_{i} = P_{i}$, and every   closed   $\sc \cup \{R\}$-invariant
subset of $C$ is also $P_{i}$-invariant.
\end{thm}
\begin{proof} Set  $ \sc ^{'}:= \sc \cup \{R\}$ and \\
 $\mathfrak{R^{'}}=\{T \in C^{C}: T \; \text{is  group nonexpansive}, T\circ T_{\alpha}=T\\
  \text{  } \text{ and every    closed   $\sc ^{'}$-invariant subset of $C$ is also $T$ -invariant}\}$,\\
and  we get that $\text{Fix}(\sc^{'})= \text{Fix}(\sc )$ and by replacing $\sc$ with $ \sc ^{'}$  and $\mathfrak{R}$ with $\mathfrak{R^{'}}$ in the proof of Theorem \ref{shlp}, we find a minimal element  $P_{\alpha}$ in the sense of ($*$). Now we have  $R\circ T \in \mathfrak{R^{'}}$ for each $T \in \mathfrak{R^{'}}$. Indeed,  $R\circ T\circ T_{\alpha}=R\circ T$  for each $T \in \mathfrak{R^{'}}$ and because $R \in \sc ^{'}$, we have that every  closed   $\sc ^{'}$-invariant subset of $C$ is also $R$ -invariant, and  therefore is $R\circ T$-invariant for each $T \in \mathfrak{R^{'}}$. Hence  for each $x \in C$, the set $K=\{T(P_{\alpha}x): T \in \mathfrak{R^{'}}\}$ is an $R$-invariant subset of $C$  for each $T \in \mathfrak{R^{'}}$. Therefore from the fact that $R(K) \subset K \cap   R(C)=K \cap \text{Fix}(\sc)$, we have $K \cap\text{Fix}(\sc^{'})=K \cap \text{Fix}(\sc)\neq \emptyset$. Now by repeating the reasoning used in Theorem \ref{shlp}, we  get the desired result.

\end{proof}
As an application of Theorem  \ref{kgfdy}, we have the following result:
\begin{thm}\label{bwopxh}
Let $G$ be a topological group with the topology $\tau$ and let $ C $   be   a      compact subset                 of $G$.
     Suppose that $\sc = \{T_{i} : i \in I \}$ is a family of  group  nonexpansive mappings on $C$ such that
$\text{Fix}(\sc)\neq \emptyset$ and for every $\alpha \in I$,   there exists a subnet $\{T_{\alpha}^{n_{\gamma}}\}$ of the sequence $\{T_{\alpha}^{n}\}$ such that $\displaystyle\lim _{\gamma}T_{\alpha}^{n_{\gamma}}x=\displaystyle\lim _{\gamma}T_{\alpha}^{n_{\gamma}-1}x$ for each $ x \in C$.
 Consider the following assumptions:
\begin{enumerate}
  \item [(a)]   Suppose
   for every nonempty    compact   $\sc $-invariant subset $K$ of $C$,
$K \cap \text{Fix}(\sc) \neq \emptyset$,
 \item [(b)]   there exists  a group nonexpansive retraction $R$ from $C$ onto
$\text{Fix}(\sc)$.
\end{enumerate}
Let   $\{P_{i}\}_{i \in I}$  be the family of retractions obtained in the above Theorem.
 Then for each $x \in C$, $$\overline{\{ T_{i}^{n}x: i \in I, n \in \mathbb{N}\}}^{\tau} \cap \text{Fix}(\sc)\subseteq \overline{\{ P_{i}(x): i \in I\}}^{\tau}.$$
\end{thm}
\begin{proof}
  Let  $g \in \overline{\{ T_{i}^{n}x: i \in I, n \in \mathbb{N}\}}^{\tau} \cap \text{Fix}(\sc)$. Then for each $ U \in \mathfrak{B}_{e} $,  there exists $i \in I$ and $ n \in \mathbb{N}$ such that $ T_{i}^{n}x g^{-1} \in U$. From our  assumptions  and using Theorems \ref{shlp}  and \ref{kgfdy}, there exists a  group nonexpansive retraction $P_{i}$ such that $P_{i}=P_{i}T_{i}$ and since from   Theorems \ref{shlp} and  \ref{kgfdy} every   closed   $\sc$-invariant  or  $\sc \cup \{R\}$-invariant
subset of $C$ is also $P_{i}$-invariant then we have $P_{i}g=g$ for each $i \in I$. Hence from the fact that $P_{i}$ is  group nonexpansive and since  $ T_{i}^{n}x g^{-1} \in U$ then we have,
$$ (P_{i}x)g^{-1} =  (P_{i}T_{i}^{n}x )(P_{i}T_{i}^{n}g)^{-1}  \in U,$$
and then we conclude $g \in \overline{\{ P_{i}(x): i \in I\}}^{\tau}$.
\end{proof}

\section{Examples and applications}
Recall every locally convex space is a topological group by the topology generated by a family of seminorms. Theorem \ref{shlp} extends and generalizes   Theorem 4.1(a) in \cite{soori}, by removing  the "convex" and  "separated" conditions     as follows:
\begin{co}
Suppose that  $Q$ is a   family of   seminorms on a    locally convex space $E$ which determines the topology of $E$. Let $ C $   be   a  compact    subset of $E$.
     Suppose that $\sc = \{T_{i} : i \in I \}$ is a family of $Q$-nonexpansive mappings on $C$ such that
$\text{Fix}(\sc)\neq \emptyset$ and for every $\alpha \in I$,   there exists a subnet $\{T_{\alpha}^{n_{\gamma}}\}$ of the sequence $\{T_{\alpha}^{n}\}$ such that $\displaystyle\lim _{\gamma}T_{\alpha}^{n_{\gamma}}x=\displaystyle\lim _{\gamma}T_{\alpha}^{n_{\gamma}-1}x$ for each $ x \in C$. If
 for every nonempty    compact   $\sc $-invariant subset $K$ of $C$,
$K \cap \text{Fix}(\sc) \neq \emptyset$,
then, for each $i \in I$, there exists a  group nonexpansive retraction $P_{i}$ from $C$ onto
$\text{Fix}(\sc)$, such that $P_{i}T_{i} = T_{i}P_{i} = P_{i}$  and every   closed   $\sc$-invariant
subset of $C$ is also $P_{i}$-invariant.
\end{co}
  Theorem \ref{shlp} extends and generalizes   Theorem 2.1(a) in \cite{said2012}, by removing the "convex"  condition      as follows:
\begin{co}
Let $ C $   be   a  compact    subset of     a Banach space $E$.
     Suppose that $\sc = \{T_{i} : i \in I \}$ is a family of  nonexpansive mappings on $C$ such that
$\text{Fix}(\sc)\neq \emptyset$ and for every $\alpha \in I$,   there exists a subnet $\{T_{\alpha}^{n_{\gamma}}\}$ of the sequence $\{T_{\alpha}^{n}\}$ such that $\displaystyle\lim _{\gamma}T_{\alpha}^{n_{\gamma}}x=\displaystyle\lim _{\gamma}T_{\alpha}^{n_{\gamma}-1}x$ for each $ x \in C$. If
 for every nonempty    compact   $\sc $-invariant subset $K$ of $C$,
$K \cap \text{Fix}(\sc) \neq \emptyset$,
then, for each $i \in I$, there exists a  group nonexpansive retraction $P_{i}$ from $C$ onto
$\text{Fix}(\sc)$, such that $P_{i}T_{i} = T_{i}P_{i} = P_{i}$  and every   closed   $\sc$-invariant
subset of $C$ is also $P_{i}$-invariant.
\end{co}

In the following example for Theorem \ref{shlp}, we present a family of retractions without assuming a convexity condition on $K$.
\begin{ex}\label{asfgb}
Let $G=\mathbb{R}$ with the  usual topology and $C=[0,2]$.
     Suppose that $\sc = \{T_{n} : n=2,3,4, \cdots \}$ is a family of the  group  nonexpansive mappings on $C$ such that
\begin{equation*}
T_{n}(x)= \left\{
       \begin{array}{ll}
         x, &  x  \in [0,\frac{2n}{2n-1}]; \\
         ( \frac{1}{n}-1)x+2, & x  \in (\frac{2n}{2n-1},2].
       \end{array}
     \right.
\end{equation*}
First note that $T_{n}^{2}x=T_{n}x$ for each  $x \in[0,2]$. Indeed if $x \in   [0,\frac{2n}{2n-1}]$   then $T_{n}^2x=T_{n}(T_{n}x)=x=T_{n}x$. Next, let $x \in (\frac{2n}{2n-1}, \frac{n}{n-1} )$ and then $ 1 <  (\frac{1}{n}-1)x+2 <\frac{2n}{2n-1} $ so $T_{n}^{2}x=T_{n}((\frac{1}{n}-1)x+2)=(\frac{1}{n}-1)x+2=T_{n}x$. Finally, if $x \in [ \frac{n}{n-1} , 2]$,  then we have $0\leq(\frac{1}{n}-1)x +2 \leq 1$, and hence    $$T_{n}^2x =T_{n}(T_{n}x)= T_{n}(( \frac{1}{n}-1)x+2)= ( \frac{1}{n}-1)x+2=T_{n}x,$$ so $T_{n}^2=T_{n}$. Hence,  $\displaystyle\lim _{m \rightarrow \infty}T_{n}^{m}x=\displaystyle\lim _{m\rightarrow \infty}T_{n}^{m-1}x= \displaystyle\lim _{m\rightarrow \infty}T_{n}x=T_{n}x$ for each  $x \in[0,2]$. Then the condition in Theorem \ref{shlp} is true.

Next we show that $T_{n}$ is group nonexpansive.  Let $x \in (\frac{2n}{2n-1} , 2]$ and $y \in [0, \frac{2n}{2n-1}]$. Note
$2(y-1)\leq \frac{2}{2n-1}\leq\frac{1}{n}x$ so $ -\frac{1}{n}x \leq -2(y-1)$,   $ -\frac{1}{n}x -2 \leq -2y$ and hence
\begin{equation*}\label{kkk}
 -2(x-y)\leq (\frac{1}{n}-2)x +2 \leq 0,
\end{equation*}
and therefore since  $x-y\geq 0$ we have
\begin{align*}
 |T_{n}(x)-T_{n}(y)|&=   |(\frac{1}{n}-1)x +2-y|  =|x-y+(\frac{1}{n}-2)x +2| \\ & \leq |x-y-2(x-y)|=|x-y|.
\end{align*}
The other  cases are easy. Hence $T_{n}$ is a group nonexpansive mapping, for each $n=2, 3, \ldots$.
Note $\text{Fix}(\sc)=[0,1]$.   Also  for
     every nonempty    compact   $\sc $-invariant subset $K$ of $C$,
$K \cap \text{Fix}(\sc) \neq \emptyset$. Indeed,  since $K$ is   $\sc $-invariant, then for each $x \in K  \cap [1,2]$, there exists a $n_{1}\in \mathbb{N}$ such that $x \in [\frac{2n_{1}}{2n_{1}-1}, 2]$  and $T_{n}(x)=( \frac{1}{n}-1)x+2 \in K$ for each $n\geq n_{1}$. Let $n\to \infty$ (note  $K$ is closed) so   $-x+2 \in K$. However  $-x+2 \in[0, \frac{2n_{1}-2}{2n_{1}-1}] \subset [0,1]$,  so $K \cap \text{Fix}(\sc) =K\cap [0,1] \neq \emptyset$. We now show that   $1\in K $. Indeed,   $-x+2 \in K$,  for each $x \in (1,2]$, because if $x \in (1,2]$ then there exists an integer    $n_{2} \in \mathbb{N}$ such that $x \in (\frac{2n_{2}}{2n_{2}-1}, 2]$ and therefore for each $n\geq n_{2}$, $T_{n}(x)=( \frac{1}{n}-1)x+2 \in K$, so let $n\to \infty$ and we have  $-x+2 \in K$ for each  $x\in(1,2]$.  Put   $x=0.1, 0.01, 0.001, \ldots $  and we get  $0.9, 0.99, 0.999, \ldots \in K$, and hence since $K$ is closed, $1\in K$.

Now define $P_{n}$ by :
\begin{equation*}
P_{n}(x)= \left\{
       \begin{array}{ll}
         x, &  x  \in [0,1]; \\
         1, &   x  \in (1,\frac{n}{n-1}); \\
         ( \frac{1}{n}-1)x+2, & x \in [\frac{n}{n-1} , 2].
                  \end{array}
     \right.
\end{equation*}
First, we show that  for each $n=2, 3,\cdots$, $P_{n}^2=P_{n}$.   Let $x \in   [0,1]$ and then $P_{n}^2x=P_{n}(P_{n}x)=P_{n}x$. Next, let $x \in (1, \frac{n}{n-1} )$ and then $P_{n}^2x=P_{n}(P_{n}x)=P_{n}(1)=1=P_{n}x$. Finally, if $x \in [ \frac{n}{n-1} , 2]$,  then we have $0\leq(\frac{1}{n}-1)x +2 \leq 1$, and hence    $$P_{n}^2x =P_{n}(P_{n}x)= P_{n}(( \frac{1}{n}-1)x+2)= ( \frac{1}{n}-1)x+2=P_{n}x,$$ so $P_{n}^2=P_{n}$.

Next we show that for each $n=2, 3,\cdots$, $P_{n}$ is a  group nonexpansive retraction  from $C$ onto
$\text{Fix}(\sc)$. If $x \in [\frac{n}{n-1} , 2]$ and $y \in (1, \frac{n}{n-1} )$,   we have
$\frac{x}{n} \geq \frac{1}{n-1}$ and then $y-\frac{x}{n}\leq 1$, so $x-\frac{x}{n}-1 \leq x-y$ and hence we get $(\frac{1}{n}-1)x+1\leq x-y$,
so $$|P_{n}(x)-P_{n}(y)|\leq |x-y|.$$  If $x \in [\frac{n}{n-1} , 2]$ and $y \in [0, 1]$, then we have
$2(y-1)\leq \frac{1}{n}x$  and then $ -\frac{1}{n}x \leq -2(y-1)$,  so $ -\frac{1}{n}x -2 \leq -2y$, and  hence
\begin{equation*}\label{kkk}
 -2(x-y)\leq (\frac{1}{n}-2)x +2 \leq 0,
\end{equation*}
and therefore since  $x-y\geq 0$ we have
\begin{align*}
 |P_{n}(x)-P_{n}(y)|&=   |(\frac{1}{n}-1)x +2-y|  =|x-y+(\frac{1}{n}-2)x +2| \\ & \leq |x-y-2(x-y)|=|x-y|.
\end{align*}
The other cases are easy. Hence $P_{n}$ is a group nonexpansive mapping.
Next to show     $P_{n}T_{n} = T_{n}P_{n} = P_{n}$.  First we prove $T_{n}P_{n} = P_{n}$.  The case  $x \in [0,1]$ is clear. Let $x \in (1, \frac{n}{n-1})$. Then  we have $T_{n}P_{n}x = T_{n}(1)=1=P_{n}x$. Finally, let  $x \in [\frac{n}{n-1} , 2]$. Then we have $\frac{2}{n}\leq(\frac{1}{n}-1)x +2 \leq1$ so  $$T_{n}P_{n}x = T_{n}((\frac{1}{n}-1)x +2)=(\frac{1}{n}-1)x +2=P_{n}x.$$ Next we show $P_{n}T_{n} = P_{n}$. Let $x \in (0, \frac{2n}{2n-1})$. Clearly  we have $P_{n}T_{n}x = P_{n}x$. Let  $x \in (\frac{2n}{2n-1} , 2]$. Then we have $$P_{n}T_{n}x = P_{n}((\frac{1}{n}-1)x +2) = P_{n}x;$$ to see this we consider two cases; (a): if $x \in (\frac{2n}{2n-1} , \frac{n}{n-1}]$, then we have $$1\leq(\frac{1}{n}-1)x +2 \leq \frac{2n}{2n-1}< \frac{n}{n-1},$$ so  $$P_{n}T_{n}x = P_{n}((\frac{1}{n}-1)x +2) = 1= P_{n}x;$$ (b): if $x \in (\frac{n}{n-1} , 2]$, then we have $0\leq(\frac{1}{n}-1)x +2 \leq 1$, so  $$P_{n}T_{n}x = P_{n}((\frac{1}{n}-1)x +2) = (\frac{1}{n}-1)x +2= P_{n}x.$$

  Now, we show that  every   closed   $\sc$-invariant
subset $K$ of $C$ is also $P_{n}$-invariant for each $n=2, 3,\cdots$. First if $x \in K\cap [0,1]$ then $P_{n}x=x \in K$. Next if $x \in K \cap (1, \frac{n}{n-1} )$ then $P_{n}x=1 \in K$. Finally, if $x \in K \cap [ \frac{n}{n-1} , 2]$, then $x \in K \cap [ \frac{2n}{2n-1} , 2]$, and hence  from the fact that $P_{n}x = ( \frac{1}{n}-1)x+2 = T_{n}x\in K$, we have  $P_{n}x \in K$. Therefore, $K$   is also $P_{n}$-invariant for each $n=2, 3,\cdots$.
\end{ex}

\begin{center}
  \bf{ Acknowledgements}
\end{center}
The third    author is  grateful to  the University of Lorestan for their support.

\end{large}
\end{document}